\documentclass[12pt]{article}
\usepackage[utf8]{inputenc}
\usepackage{amsthm}
\usepackage{amsmath}
\usepackage{bbm}
\usepackage{amsfonts}
\usepackage{amssymb}
\usepackage[left=2.5cm,right=2.5cm,top=2cm,bottom=2cm]{geometry}

\usepackage{tabto}
\TabPositions{2 cm, 4 cm, 6 cm, 8 cm}

\usepackage{tikz-cd}

\usepackage{tikz}
\usetikzlibrary{arrows, automata, positioning}

\usepackage{tocloft}
\usepackage{appendix}

\usepackage{enumerate}

\usepackage{hyperref}
 
\newtheorem{theorem}{Theorem}[section]
\newtheorem{corollary}[theorem]{Corollary}
\newtheorem{lemma}[theorem]{Lemma}

\newtheorem{proposition}[theorem]{Proposition}

\theoremstyle{definition}
\newtheorem{definition}[theorem]{Definition}
\newtheorem{example}[theorem]{Example}
\newtheorem*{definition*}{Definition}
\newtheorem{remark}[theorem]{Remark}

\newtheorem*{lemma*}{Lemma}
\newtheorem*{proposition*}{Proposition}
\newtheorem*{theorem*}{Theorem}
\newtheorem*{corollary*}{Corollary}

\theoremstyle{definition}

\newtheorem{question}[theorem]{Question}

\newcommand{\Homeo}{\operatorname{Homeo}}

\newcommand{\cl}{\operatorname{cl}}
\newcommand{\scl}{\operatorname{scl}}
\newcommand{\rot}{\operatorname{rot}}

\begin{document}

\title{Algebraic irrational stable commutator length \\ in finitely presented groups}
\author{Francesco Fournier-Facio and Yash Lodha}
\date{\today}
\maketitle

\begin{abstract}
We provide the first example of a finitely presented (in fact, type $F_\infty$) group with elements whose stable commutator length is algebraic and irrational, answering a question of Calegari. Our example is the lift to the real line of the \emph{golden ratio Thompson group} $T_\tau$: the circle analogue of the Cleary's golden ratio Thompson group $F_\tau$ which acts on the interval.
\end{abstract}

\section{Introduction}

Stable commutator length (usually referred to as $\scl$) is a real invariant of groups that allows for algebraic, topological, and analytic characterizations, and has several applications in topology and geometry \cite{calegari}. 
Consider a group $G$. The \emph{commutator length} of an element $g$ in the commutator subgroup $G'$, denoted as $\cl(g)$, is the smallest $n\in \mathbf{N}$ such that $g$ can be expressed as a product of at most $n$ commutators of elements in $G$. Then we define the \emph{stable commutator length} of $g$ as $\scl(g) := \lim_{n\to \infty} \frac{\cl(g^n)}{n}$.
If there exists $k \geq 1$ such that $g^k \in G'$, then we set $\scl(g) := \frac{\scl(g^k)}{k}$. Otherwise, we set $\scl(g) := \infty$.
While it is not too hard to show that the limit above always exists, in general it is hard to compute. 
Given a group $G$, we denote $\scl(G) := \{\scl(g)\mid g\in G\}$.\\

Given a class $\mathcal{G}$ of groups, the \emph{spectrum of scl on} $\mathcal{G}$ is defined as $\{r\in \mathbf{R}_{\geq 0} \mid \exists G\in \mathcal{G}\text{ such that }r\in \scl(G)\}$.
A question that has received a lot of recent attention is to compute this on certain classes of groups. For the classes of finitely generated groups, and finitely generated recursively presented groups, this question has been completely answered by Heuer in \cite{heuer}. However the full picture for finitely presented groups remains mysterious \cite[Question 5.48]{calegari}. \\

It has been shown for several classes of groups that the spectrum of $\scl$ consists entirely of rationals. In fact, the lift to the real line of Thompson's group $T$ is a group of type $F_\infty$ (so in particular finitely presented) whose $\scl$ spectrum coincides with the positive rationals \cite{ghys_serg}.

For other classes of groups, such as free \cite{free} and Baumslag--Solitar groups \cite{gog}, the computation of $\scl$ is more involved, and its rationality is proven using algorithms which rely on the topological definition of $\scl$ via surfaces \cite[Proposition 2.10]{calegari}. This definition involves an infimum of rational numbers, and a feature of these algorithms is that in some cases this infimum is realized by certain surfaces, called \emph{extremal}. \\

On the other hand, if an element has irrational $\scl$, then extremal surfaces cannot exist. In \cite[6 C]{gromov}, Gromov asked whether this situation can occur in finitely presented groups. The first examples were found by Zhuang \cite{zhuang}, who showed that lifts of certain Stein--Thompson's groups have elements with transcendental $\scl$. Since then other examples have emerged \cite[Chapter 5]{calegari}, but the $\scl$ always appears to be either rational or transcendental. This leads to the following natural question:

\begin{question}[Calegari {\cite[Question 5.47]{calegari}}]
\label{q}

Is there a finitely presented group in which $\scl$ takes on an irrational value that is algebraic?
\end{question}

This question was also asked in \cite{book} and mentioned in \cite{heuer}. Here we answer it in the positive:

\begin{theorem}
\label{thm:main}

There exists a finitely presented (in fact, type $F_\infty$) group $G$ such that $\scl(G)$ contains all of $\frac{1}{2} \mathbf{Z}[\tau]$, where $\tau = \frac{\sqrt{5} - 1}{2}$ is the small golden ratio.
\end{theorem}

The group $G$ is the lift to the real line of the \emph{golden ratio Thompson's group} $T_\tau$, which is the natural ``circle version" of the group $F_{\tau}$ defined by Cleary in \cite{cleary}. 
(Just like Thompson's group $T$ is the ``circle version" of Thompson's group $F$.)
The group $T_{\tau}$ differs from $T$ for two main reasons: it contains irrational rotations (while every element in $T$ has a periodic point) and its commutator subgroup has index two and is simple (while $T$ itself is simple).

The group $T_\tau$ was studied in \cite{burilloT} from a combinatorial viewpoint; here we will treat it dynamically, since it is its realization as a group of homeomorphisms of the circle that establishes the connection with $\scl$. Our approach is analogous to that of \cite{zhuang}, in that we reduce the computation of $\scl$ to a computation of rotation numbers in $T_\tau$. On the way we also prove that $T_\tau$ is of type $F_\infty$ and that $T_\tau'$ is uniformly simple, strengthening the results from \cite{burilloT}. \\

\textbf{Acknowledgements:} The first author was supported by an ETH Z\"urich Doc.Mobility Fellowship. The second author was supported by a START-preis grant of the Austrian science fund. The authors would like to thank Matt Brin and Justin Moore for their helpful comments, and Nicolaus Heuer for introducing them to this question and to stable commutator length.

\section{Preliminaries}

Group actions will always be on the right. Accordingly, we will use the convention $[g, h] = g^{-1} h^{-1} g h$ for commutators.

\subsection{Dynamics of group actions on $1$-manifolds}

Given $g\in \Homeo^+(M)$ for a given $1$-manifold $M$, we define the \emph{support of $g$} as: $$Supp(g) := \{x\in M\mid x\cdot g\neq x\}.$$
We identify $\mathbf{S}^1$ with $\mathbf{R}/\mathbf{Z}$.
Recall that $\Homeo^+(\mathbf{R})$ and $\Homeo^+(\mathbf{R/Z})$ are the groups of orientation preserving homeomorphisms of the real line and the circle. The latter group action admits a so-called \emph{lift} to the real line, which is the group $\overline{\Homeo^+(\mathbf{R}/\mathbf{Z})}\leq \Homeo^+(\mathbf{R})$ which equals the full centralizer of the group of integer translations $\mathbf{Z}=\{x\mapsto x+n\mid n\in \mathbf{Z}\}$ inside $\Homeo^+(\mathbf{R})$. There is a short exact sequence $$1\to \mathbf{Z}\to \overline{\Homeo^+(\mathbf{R}/\mathbf{Z})} \to \Homeo^+(\mathbf{R}/\mathbf{Z}) \to 1.$$
In this paper, we fix the notation for the map $$\eta: \overline{\Homeo^+(\mathbf{R}/\mathbf{Z})}\to \Homeo^+(\mathbf{R}/\mathbf{Z}).$$

Given a group $G\leq \Homeo^+(\mathbf{R}/\mathbf{Z})$, we define the \emph{total lift}
of $G$ to be the group $\eta^{-1}(G)$, which is a central extension of $G$ by $\mathbf{Z}$. Throughout this paper, we will denote the total lift of $G$ by $\overline{G}$.
More generally, we may define a \emph{lift} of $G$ to $\mathbf{R}$ as a group action $H\leq \overline{\Homeo^+(\mathbf{R}/\mathbf{Z})}$ satisfying that $\eta(H)=G$. Such a lift may not be unique.
On the other hand, the total lift of $G$ is the group generated by all possible such lifts.

A group action $G\leq \Homeo^+(\mathbf{R})$ is said to be \emph{boundedly supported} if for each $f\in G$ there is a compact interval $I\subset \mathbf{R}$ such that for all $x \in \mathbf{R}\setminus I$ it holds that $x\cdot f=x$.
A group action $G\leq \Homeo^+(\mathbf{R})$ is said to be \emph{proximal}, if for each pair of nonempty open sets $I,J\subset \mathbf{R}$ satisfying that $J$ is bounded, there is an element $f\in G$ such that $J\cdot f\subset I$. \\

Recall that a group $G$ is \emph{$n$-uniformly simple} for a fixed $n\in \mathbf{N}$, if for each pair of elements $f,g\in G\setminus \{id\}$, $f$ can be expressed as a product of at most $n$ conjugates of the elements $g,g^{-1}$.
$G$ is said to be \emph{$n$-uniformly perfect} if each element $f\in G\setminus \{id\}$
can be expressed as a product of at most $n$ commutators of elements in $G$. 
A group is uniformly simple (or uniformly perfect) if there exists an $n\in \mathbf{N}$ such that the group is $n$-uniformly simple (respectively, $n$-uniformly perfect).
It is easy to see that uniformly simple groups are uniformly perfect.
The following is a special case of Theorem $1.1$ in \cite{prox}.

\begin{theorem}[Gal, Gismatullin \cite{prox}]
\label{thm:uniformly simple}
Let $G\leq \Homeo^+(\mathbf{R})$ be a boundedly supported and promixal group action.
Then $G'$ is $6$-uniformly simple and $2$-uniformly perfect.
\end{theorem}

\subsection{Stable commutator length}

For a comprehensive survey on the topic, we refer the reader to \cite[Chapter 2]{calegari}. The following is a direct consequence of the definition of stable commutator length as stated in the introduction.

\begin{lemma}
Let $G$ be a uniformly perfect group. Then $\scl$ vanishes everywhere on $G$.
\end{lemma}

The way we will compute $\scl$ is via quasimorphisms.

\begin{definition}
Let $\phi : G \to \mathbf{R}$. Its \emph{defect} is
$$D(\phi) := \sup\limits_{g, h \in G} \left|\phi(g) + \phi(h) - \phi(gh)\right|.$$
If the defect is finite, we say that $\phi$ is a \emph{quasimorphism}. If moreover $\phi(g^k) = k \phi(g)$ for every $g \in G$ and every $k \in \mathbf{Z}$, we say that $\phi$ is \emph{homogeneous}.
\end{definition}

Every quasimorphism is at a bounded distance from a unique homogeneous quasimorphism. The following example is the fundamental one for our purposes:

\begin{example}[Poincar\'e, see {\cite[11.1]{ds}}]
\label{ex:rot}

Let $g \in \overline{\Homeo^+(\mathbf{R}/\mathbf{Z})}$ be an orientation-preserving homeomorphism of the line that commutes with integer translations. Define the \emph{rotation number} of $g$ to be
$$\rot(g) := \lim\limits_{n \to \infty} \frac{0 \cdot g^n}{n}.$$
Then $\rot : \overline{\Homeo^+(\mathbf{R}/\mathbf{Z})} \to \mathbf{R}$ is a homogeneous quasimorphism of defect $1$ \cite[Proposition 2.92]{calegari}.

It follows that for every subgroup $G \leq \overline{\Homeo^+(\mathbf{R}/\mathbf{Z})}$, the restriction of $\rot$ to $G$ is a homogeneous quasimorphism of defect at most $1$. Note however that the defect of $\rot|_G$ need not be equal to $1$: for example $\rot|_{\mathbf{Z}}$ is a homomorphism, so its defect is $0$.
\end{example}

The connection between $\scl$ and quasimorphisms is provided by the following fundamental result (see \cite[Theorem 2.70]{calegari}):

\begin{theorem}[Bavard Duality \cite{bavard}]
\label{thm:bavard}

Let $G$ be a group and let $g \in G'$. Then:
$$\scl(g) = \frac{1}{2} \sup \frac{|\phi(g)|}{D(\phi)};$$
where the supremum runs over all homogeneous quasimorphisms of positive defect.
\end{theorem}

\begin{corollary}[Bavard \cite{bavard}]
\label{cor:bavard}

Let $G$ be a group. Then $\scl$ vanishes everywhere on $G'$ if and only if every homogeneous quasimorphism on $G$ is a homomorphism.
\end{corollary}

In particular, if $G$ is such that the space of homogeneous quasimorphisms is one-dimensional and spanned by $\phi$, then the computation of $\scl$ reduces to the evaluation of $\phi$ and its defect. This is the approach taken in \cite{zhuang} to produce examples of groups with transcendental $\scl$. The general statement for groups acting on the circle is the following:

\begin{theorem}[Zhuang \cite{zhuang}]
\label{thm:zhuang}

Let $G \leq \Homeo^+(\mathbf{R}/\mathbf{Z})$, let $\overline{G} \leq \overline{\Homeo^+(\mathbf{R}/\mathbf{Z})}$ be its total lift to the real line, and let $D := D(\rot|_{\overline{G}})$. Suppose that $\scl$ vanishes everywhere on $G$. Then either $D = 0$, or
$$\scl(g) = \frac{|\rot(g)|}{2D}$$
for every $g \in \overline{G}'$.
\end{theorem}

Although this result is implicit in \cite{zhuang}, it is never stated as a theorem and is only mentioned briefly and informally. For completeness, we provide a proof below.

\begin{proof}
We start by showing that no non-trivial homogeneous quasimorphism vanishes on the unit translation $x \mapsto (x + 1)$: the proof is similar to the analogous statement for $\overline{\Homeo^+(\mathbf{R}/\mathbf{Z})}$ \cite{bg}.
Indeed, suppose that $\phi : \overline{G} \to \mathbf{R}$ is such a quasimorphism. Homogeneity implies that $\phi$ vanishes on the whole central subgroup $\mathbf{Z} \leq \overline{G}$. Now every homogeneous quasimorphism on an abelian group is a homomorphism \cite[Proposition 2.65]{calegari}, so if $g \in \overline{G}$ and $z \in \mathbf{Z}$, then $\phi(gz) = \phi(g) + \phi(z) = \phi(g)$. It follows from this that $\phi$ induces a well-defined quasimorphism $\psi : G \to \mathbf{R}$. By Corollary \ref{cor:bavard}, together with the vanishing assumption, every homogeneous quasimorphism on $G$ is a homomorphism. Moreover, since every element of $G$ has finite $\scl$, the abelianization of $G$ is torsion, so $G$ admits no non-trivial homomorphism to $\mathbf{R}$ either. So $\psi$ must vanish everywhere on $G$, and therefore $\phi$ vanishes everywhere on $\overline{G}$. \\

This also implies that $\rot$ is the unique homogeneous quasimorphism on $\overline{G}$ that sends the unit translation to $1$: if $\phi$ is any other such quasimorphism, then $(\rot - \phi)$ vanishes on the unit translation and so it is trivial. These two facts together imply that the space of homogeneous quasimorphisms of $\overline{G}$ is one-dimensional and spanned by $\rot$. The formula then follows from Bavard Duality (Theorem \ref{thm:bavard}).
\end{proof}

\begin{remark}
A more conceptual proof using bounded cohomology can be found in \cite[Theorem 5.16]{calegari}. There too the general statement is only implicit, but the proof carries over to this setting.
\end{remark}

\subsection{Finiteness properties}

A group $G$ is said to be \emph{of type $F_n$} for an integer $n \geq 1$ if it admits a classifying space with compact $n$-skeleton, and \emph{of type $F_\infty$} if it is of type $F_n$ for all $n \geq 1$. Being of type $F_1$ is equivalent to being finitely generated, and being of type $F_2$ is equivalent to being finitely presented.


These properties are preserved by extensions, so in particular a central extension by $\mathbf{Z}$ of a group of type $F_\infty$ is still of type $F_\infty$.




We will use the following special case of Brown's finiteness criterion to show that the groups we are interested in are of type $F_\infty$.

\begin{proposition}[Brown {\cite[Proposition 1.1]{brown_fp}}]
\label{prop:brown}
Let $X$ be a simplicial complex, and let $G$ be a group acting on $X$ by simplicial automorphisms. Suppose that
\begin{enumerate}
    \item $X$ is contractible;
    \item $G$ acts cocompactly on simplices of every given dimension;
    \item The stabilizer of a simplex is of type $F_\infty$.
\end{enumerate}
Then $G$ is of type $F_\infty$.
\end{proposition}

We refer the reader to \cite{brown_book, brown_fp} and \cite{alonso} for more details.

\section{The golden ratio Thompson group}

Let $\tau := \frac{\sqrt{5} - 1}{2}$ be the small golden ratio. In \cite{cleary}, the author defines the \emph{golden ratio Thompson group} $F_\tau$ as the group of orientation preserving piecewise linear homeomorphisms $g$ of $[0, 1]$ such that the following two properties are satisfied:
\begin{enumerate}
    \item $g$ has finitely many breakpoints, all of which lie in $\mathbf{Z}[\tau]$;
    \item The slope of $g$, whenever it is defined, is a power of $\tau$.
\end{enumerate}

Note that it automatically follows that $F_\tau$ preserves $\mathbf{Z}[\tau] \cap [0, 1]$. This is the analogous definition as Thompson's group $F$, where $\tau$ plays the role of $\frac{1}{2}$. 
This group was studied combinatorially via tree diagrams in \cite{burilloF}. The group action on $\mathbf{R/Z}$ corresponding to the tree diagrams defined in \cite{burilloF} coincides with the above action.
\\

We consider the following ``circle" analogue of the group $F_\tau$.

\begin{definition}\label{Definition: main}
Define the \emph{golden ratio Thompson group} $T_\tau$ as the group of orientation preserving piecewise linear homeomorphisms $g$ of $\mathbf{S}^1 = \mathbf{R}/\mathbf{Z}$ such that the following three properties are satisfied:

\begin{enumerate}
    \item $g$ has finitely many breakpoints, all of which lie in in $\mathbf{Z}[\tau]/\mathbf{Z}$;
    \item The slope of $g$, whenever it is defined, is a power of $\tau$;
    \item $g$ preserves $\mathbf{Z}[\tau]/\mathbf{Z}$.
\end{enumerate}
\end{definition}

Note that, unlike with $F_\tau$, it is necessary here to impose the third condition, since every rotation satisfies the first two. Again, this is the analogous definition as Thompson's group $T$, where $\tau$ plays the role of $\frac{1}{2}$. Crucially, $T_\tau$ contains all rotations with angles in $\mathbf{Z}[\tau]$, and the stabilizer of a point $x \in \mathbf{Z}[\tau]/\mathbf{Z}$ is naturally isomorphic to $F_\tau$. \\

The group $T_{\tau}$ has been studied in \cite{burilloT} by combinatorial means.
We now clarify that the definition of $T_{\tau}$ provided in \cite{burilloT}, by means of combinatorial tree diagrams, defines the same group as Definition \ref{Definition: main}. Note that the group action on $\mathbf{R/Z}$ defined as $T_{\tau}$ in page $2$ of \cite{burilloT} is incorrect, since the third condition of our Definition \ref{Definition: main} is omitted. However, the combinatorial model described by tree diagrams in \cite{burilloT} is correct and it corresponds to our definition above.

\begin{lemma}\label{lemma: same group}
The group action of $T_{\tau}$ on $\mathbf{R/Z}$ in Definition \ref{Definition: main} and the group action on $\mathbf{R/Z}$ corresponding to the group $T_{\tau}$ defined by combinatorial tree diagrams in Section $2$ of \cite{burilloT} represent the same group action.
\end{lemma}

\begin{proof}
For the sake of this proof, we denote the group action defined in Section $2$ of \cite{burilloT} by means of combinatorial tree diagrams as $G$, and we will show that it is the same as $T_{\tau}$ above. Note that the combinatorial tree diagrams defined in \cite{burilloT} readily translate into concrete piecewise linear homeomorphisms of $\mathbf{R/Z}$ and we will use the latter without giving all details. 

The elements of the generating set $\{x_n,y_n,c_n\}$ of $G$ (as provided in Section $2$ of \cite{burilloT}) are easily seen to satisfy Definition \ref{Definition: main}, so it follows that $G\leq T_{\tau}$.
Now note that the orbit of $0$ in both group actions $G$ and $T_{\tau}$ is the same since it equals $\mathbf{Z}[\tau]/\mathbf{Z}$. So given an element $f\in T_{\tau}$, there is an element $g\in G$ such that $0\cdot fg^{-1}=0$. In particular, $fg^{-1}\in F_{\tau}$. 

The group action on $\mathbf{R/Z}$ corresponding to the tree diagrams defined in \cite{burilloF} representing elements of $F_{\tau}$ coincides with our definition of $F_{\tau}$, which is also the stabilizer of $0$ in $G$.
It follows that $F_{\tau}\leq G$ and thus $f\in G$. Therefore, $T_{\tau}\leq G$ and $G=T_{\tau}$.
\end{proof}

Our result will make use of the following structural results from \cite{burilloF, burilloT}. Let $F_\tau^c \leq F_\tau$ be the subgroup of elements $g$ for which $\overline{Supp(g)}\subset (0, 1)$.

\begin{proposition}[Burillo--Nucinkis--Reeves {\cite[Proposition 5.2]{burilloF}}]
\label{prop:index2F}
$F_\tau' = (F_\tau^c)'$ is an index-two subgroup of $F_\tau^c$.
\end{proposition}

\begin{proposition}[Burillo--Nucinkis--Reeves {\cite[Theorem 3.2]{burilloT}}]
\label{prop:index2T}

$T_\tau'$ is an index-two subgroup of $T_\tau$.
\end{proposition}

In this section we study the homological, algebraic and dynamical properties of $T_\tau$ that will be needed to prove Theorem \ref{thm:main}.

\subsection{Dynamical properties}

In this subsection we collect some useful facts about the dynamics of the action of $T_\tau$ on $\mathbf{R}/\mathbf{Z}$. These will be used to prove that $T_\tau'$ is uniformly simple. Note that in order to apply Theorem \ref{thm:zhuang}, simplicity of $T_{\tau}'$ from Proposition \ref{prop:index2T} is not sufficient, and uniform simplicity is needed.\\

We now record a few simple facts about the transitivity properties of $F_\tau$.
For every $a, b \in \mathbf{Z}[\tau] \cap [0, 1]$, let $F_\tau[a, b]$ be the set of elements whose support is contained in $[a, b]$.
It is shown in \cite[Proposition 6.2]{burilloF} that $F_\tau[a, b]$ is naturally isomorphic to $F_\tau$. (In fact, the two actions are topologically conjugate: the conjugating map is a carefully chosen piecewise linear map $[a, b] \to [0, 1]$).

\begin{lemma}
\label{lem:transF}
The action of $F_\tau$ on the set of ordered $n$-tuples $0 < x_1 < \cdots < x_n < 1$ in $\mathbf{Z}[\tau]\cap [0,1]$ is transitive, and the stabilizer of each such $n$-tuple is isomorphic to $F_\tau^{n+1}$. 
\end{lemma}

\begin{proof}
  Since for each $a,b\in \mathbf{Z}[\tau] \cap [0, 1]$, $F_\tau[a, b]$ is naturally isomorphic to $F_\tau$, this implies the statement about the stabilizers. Indeed, by the self similarity feature of the definition, the stabilizer of $\{x_1, \ldots, x_n\}$ splits as a direct product $F_\tau[0, x_1] \times \cdots \times F_\tau[x_n, 1]$.

Double transitivity is proven in \cite[Corollary 1]{cleary}. Finally, high transitivity follows from an elementary inductive argument using transitivity and the previous description of the stabilizers.
\end{proof}

\begin{lemma}\label{lem:nTransPrime}
For each $n\in \mathbf{N}$, the action of $F_{\tau}'$ on ordered 
$n$-tuples in $\mathbf{Z}[\tau] \cap (0,1)$ is transitive. 
\end{lemma}
\begin{proof}
Let $0 < x_1 < \cdots < x_n < 1$ and $0 < y_1 < \cdots < y_n < 1$.
We choose $a \in (0, \min\{x_1, y_1\}) \cap \mathbf{Z}[\tau]$ and $b \in (\max\{x_n, y_n\}, 1) \cap \mathbf{Z}[\tau]$.

By Lemma \ref{lem:transF}, there exists an element $g \in F_\tau$ fixing $a$ and $b$ and sending $x_i$ to $y_i$ for $i = 1, \ldots, n$. Let $h \in F_\tau$ be supported on $[0, a] \cup [b, 1]$ so that
$$g \mid [0, a] \cup [b, 1] = h \mid [0, a] \cup [b, 1].$$
Then $f := gh^{-1} \in F_\tau^c$ fixes $[0, a] \cup [b, 1]$ pointwise and sends $x_i$ to $y_i$, for $i = 1, \ldots, n$.

Using proximality, we find an element $l \in F_\tau$ such that $[a, b]\cdot l\subset (b, 1)$. Then $l^{-1}f^{-1}l$ is supported on $[b, 1]$, and in particular it fixes every $x_i$. Therefore $[l, f] \in F_\tau'$ is the desired element sending $x_i$ to $y_i$ for all $i = 1, \ldots, n$.
\end{proof}

\begin{corollary}
\label{cor:Fus}
The group $F_\tau'$ is $6$-uniformly simple, and $2$-uniformly perfect.
\end{corollary}

\begin{proof}
The action of $F_{\tau}^c$ on $(0,1)$ is proximal by an application of Lemma \ref{lem:nTransPrime}, since $F_{\tau}'\subset F_{\tau}^c$.
Then we apply Theorem \ref{thm:uniformly simple} to $F_{\tau}^c$ to conclude that $(F_{\tau}^c)'$ satisfies the stated properties. Since $F_{\tau}'=(F_{\tau}^c)'$ by Proposition \ref{prop:index2F}, we are done.
\end{proof}

Let us move to $T_\tau$ and prove the analogous properties.

\begin{lemma}
\label{lem:transT}
The action of $T_\tau$ on the set of circularly ordered $n$-tuples in $\mathbf{Z}[\tau]/\mathbf{Z}$ (i.e., inheriting the natural circular ordering)  is transitive, and the stabilizer is isomorphic to $F_\tau^n$.
\end{lemma}

\begin{proof}
Since $T_\tau$ contains all rotations by an angle in $\mathbf{Z}[\tau]/\mathbf{Z}$, the action on $\mathbf{Z}[\tau]/\mathbf{Z}$ is transitive. Moreover, the stabilizer of each point $x\in \mathbf{Z}[\tau]/\mathbf{Z}$ is naturally isomorphic to $F_\tau$, so high transitivity now follows from Lemma \ref{lem:transF}.
\end{proof}

Given $x \in \mathbf{Z}[\tau]/\mathbf{Z}$, we define
$$T_\tau(x) := \{f\in T_{\tau}\mid \exists \text{ an open interval }I\text{ containing $x$ such that }f\in PStab_{T_{\tau}}(I)\};$$
where
$$PStab_{T_{\tau}}(I)=\{f\in T_{\tau}\mid \forall x\in I, x\cdot f=x\}.$$

Under the conjugacy which identifies the stabilizer of $x$ in $T_{\tau}$ with $F_\tau$, the subgroup $T_\tau(x)$ is naturally topologically conjugate to $F_\tau^c$. In particular, by Proposition \ref{prop:index2F} and Corollary \ref{cor:Fus}, the commutator subgroup of $T_\tau(x)$ has index $2$ in $T_{\tau}(x)$ and it is $6$-uniformly simple. \\

The following factorization lemma is the key step towards establishing uniform simplicity of $T_\tau'$ using uniform simplicity of $F_\tau'$. As in the case of $F_\tau$, for every $a, b \in \mathbf{Z}[\tau]/\mathbf{Z}$ we let $T_\tau[a, b]$ be the set of elements whose support is contained in the closed arc $[a, b]$.

\begin{lemma}
\label{lem:Tus}
Let $g \in T_{\tau}\setminus \{id\}$ be an element. Then there are $x\neq y\in\mathbf{Z}[\tau]/\mathbf{Z}$ such that $g \in T_{\tau}(x) T_{\tau}(y)'$. 
\end{lemma}

\begin{proof}
 Let  $x \neq y \in \mathbf{Z}[\tau]/\mathbf{Z}$ be such that 
 $x \cdot g \notin \{x, y\}$.
Let $I=[a,b]$ be an arc with endpoints in $\mathbf{Z}[\tau]/\mathbf{Z}$ such that $x \in int(I), y \notin I$ and $I \cdot g$ does not contain $x$ or $y$. Recall that the action of $F_{\tau}'$ on $(0, 1)$ is $2$-transitive, from Lemma \ref{lem:nTransPrime}. Then the natural isomorphism 
$T_{\tau}(y)' \cong F_{\tau}'$ provides an element $f \in T_{\tau}(y)'$ such that $I \cdot g = I \cdot f$. 
Let $h_1\in T_{\tau}[a,b]\leq T_{\tau}(y)$ be such that 
$h_1\mid I=gf^{-1}\mid I$.
By proximality, we find an element $h_2\in T_{\tau}(y)$ such that $supp(h_1)\cdot h_2\cap supp(h_1)=\emptyset$.
Then $h_3=[h_2,h_1]\in T_{\tau}(y)'$ satisfies that $$gf^{-1}h_3^{-1}\mid I=id\mid I.$$
In other words, $gf^{-1}h_3^{-1}$ fixes $I$ pointwise, so $$gf^{-1}h_3^{-1} \in T_{\tau}(x),\qquad  h_3f\in T_{\tau}(y)',\qquad (gf^{-1}h_3^{-1})(h_3f)=g.$$
\end{proof}

\begin{lemma}
\label{lem:comm}

Let $x \in \mathbf{Z}[\tau]/\mathbf{Z}$. Then $T_\tau' \cap T_\tau(x) = T_\tau(x)'$.
\end{lemma}

\begin{proof}
It is clear that $T_{\tau}(x)'\subseteq T_\tau' \cap T_\tau(x)$.
We must show that $T_\tau' \cap T_\tau(x) \subseteq T_\tau(x)'$.
Since $T_{\tau}'$ is an index $2$ subgroup of $T_{\tau}$, it holds that the index of $T_{\tau}'\cap H$ in $H$ of any subgroup $H\leq T_{\tau}$ is at most $2$. 
It follows that $T_\tau' \cap T_\tau(x)$ has index at most $2$ in $T_{\tau}(x)$. We first claim that this index must be precisely equal to $2$. If not, it would hold that $T_{\tau}(x)\subset T_{\tau}'$, and thus $T_{\tau}(y)\subset T_{\tau}'$ for every $y \in \mathbf{Z}[\tau]/\mathbf{Z}$. Since $T_{\tau}$ is generated by the subgroups $\{T_{\tau}(y)\mid y\in \mathbf{Z}[\tau]/\mathbf{Z}\}$ (by an application of Lemma \ref{lem:Tus}),
this would imply that $T_{\tau}'=T_{\tau}$, contradicting Proposition \ref{prop:index2T}.

Since both $T_{\tau}(x)'$ and $T_\tau' \cap T_\tau(x)$ have index $2$ in $T_{\tau}(x)$, and $T_{\tau}(x)'\subseteq T_\tau' \cap T_\tau(x)$, we conclude that they must coincide.
\end{proof}

\begin{corollary}\label{cor: generation}
$\langle \{T_{\tau}(x)'\mid x\in \mathbf{Z}[\tau]/\mathbf{Z}\}\rangle=T_{\tau}'$.
\end{corollary}

\begin{proof}
This follows immediately from Lemmas \ref{lem:Tus} and \ref{lem:comm}.
\end{proof}
We are now ready to prove the main result of this subsection.

\begin{proposition}
\label{prop:Tus}
$T_\tau'$ is uniformly simple.
\end{proposition}

\begin{proof}
We will show that for any pair $f,g\in T_{\tau}'\setminus \{id\}$, $f$ is a product of at most $24$ elements in the union of the conjugacy classes of $g,g^{-1}$.

By Lemma \ref{lem:Tus}, we know that there are $x\neq y\in\mathbf{Z}[\tau]/\mathbf{Z}$ such that $f=f_1f_2$ where $f_1\in T_{\tau}(x)$ and $f_2\in T_{\tau}(y)'$. Note that by our assumption that $f\in T_{\tau}'$, it must hold that $f_1\in T_{\tau}'$.
By Lemma \ref{lem:comm}, since $f_1\in T_{\tau}'\cap T_{\tau}(x)$, it follows that $f_1\in T_{\tau}(x)'$. \\

The result will follow once we have proved the following claim.

{\bf Claim}: Each $f_i$ is a product of at most $12$ elements in the union of the conjugacy classes of $g,g^{-1}$.

{\bf Proof of claim}: It suffices to show this for $f_1$, the proof for $f_2$ is similar.

Let $I\subset \mathbf{R/Z}$ be a closed interval with nonempty interior such that:
\begin{enumerate}
\item $I\cdot g\cap I=\emptyset$.
\item $x\notin (I\cup (I\cdot g))$.
\end{enumerate}
Let $h\in T_{\tau}(x)'$ be an element such that $Supp(h)\subset int(I)$.
It follows that $$k=[g,h]\in T_{\tau}(x)\cap T_{\tau}'=T_{\tau}(x)'.$$
Since $F_{\tau}'$ is $6$-uniformly simple by Corollary \ref{cor:Fus}, we know that $T_{\tau}(x)'$ is also $6$-uniformly simple. Since $f_1\in T_{\tau}(x)'$, we conclude that $f_1$ is a product of at most $6$ elements of the union of the conjugacy classes of $k,k^{-1}$.
Since $k,k^{-1}$ are products of two elements of the union of the conjugacy classes of $g,g^{-1}$, we conclude the proof.
\end{proof}

\subsection{Finiteness properties}

The goal of this subsection is to show the following:

\begin{proposition}
\label{prop:Tfin}

The group $T_\tau$ is of type $F_\infty$.
\end{proposition}

Finite presentability (i.e., type $F_2$) was already proven in \cite{burilloT}, where moreover a concrete presentation is provided. We will prove Proposition \ref{prop:Tfin} using Brown's finiteness criterion (Proposition \ref{prop:brown}). The starting point is the analogous result of Cleary for $F_\tau$:

\begin{proposition}[Cleary \cite{cleary}]
\label{prop:Ffin}

The group $F_\tau$ is of type $F_\infty$.
\end{proposition}

\begin{proof}[Proof of Proposition \ref{prop:Tfin}]
We construct a simplicial complex $X$ as follows: the $n$-simplices are given by $(n+1)$-tuples $(x_0, \ldots, x_n)$ of elements in $\mathbf{Z}[\tau]/\mathbf{Z}$, and the face relation is given by inclusion. This is clearly contractible, and endowed with a natural action of $T_\tau$. In order to apply Brown's finiteness criterion (Proposition \ref{prop:brown}) we need to show that the action is cocompact on simplices of every given dimension, and that the stabilizer of a simplex is of type $F_\infty$.

Recall from Lemma \ref{lem:transT} that $T_\tau$ acts transitively on the set of circularly ordered $n$-tuples in $\mathbf{Z}[\tau]/\mathbf{Z}$. It follows that the action the $n$-simplices of $X$ has at most $n!$ orbits, and thus it is cocompact.
By Lemma \ref{lem:transT} again, the stabilizer of an $n$-simplex is isomorphic to $F_\tau^{n+1}$. In particular it is of type $F_\infty$ since this is a property closed under finite direct sums.
\end{proof}

\subsection{The lift of $T_\tau$}

We now move to the total lift $G := \overline{T_\tau}$ of $T_\tau$, which will be the group in Theorem \ref{thm:main}. As a first step, we need to show that $G$ also has abelianization $\mathbf{Z}/2\mathbf{Z}$. To this end, we consider the total lift $H := \overline{T_\tau'}$ of $T_\tau'$.

\begin{lemma}\label{lemma: lift}
The total lift $G$ is the unique lift of $T_\tau$ to the real line. 
The same holds for the total lift $H$ of $T_\tau'$. 
Moreover, it holds that $G'=H$ and $[G:G']=2$.
\end{lemma}

\begin{proof}
We will show the first statement for $H$, a similar argument applies to $G$.
Consider an arbitrary lift $K\leq \overline{\Homeo^+(\mathbf{R/Z})}$ of $T_{\tau}'$.
To show that it is equal to $H$, it suffices to show that $K$ must contain all integer translations.

Recall the map $\eta:\overline{\Homeo^+(\mathbf{R/Z})}\to \Homeo^+(\mathbf{R/Z})$.
Let $x\in \mathbf{R/Z}$.
Let $f_1,f_2\in T_{\tau}(x)'$ and let $\lambda_1,\lambda_2\in K$ be elements such that $\eta(\lambda_i)=f_i$. Since each $f_i$ fixes $x$, each $\lambda_i$ preserves $x+\mathbf{Z}$ in $\mathbf{R}$.
It follows that the element $\lambda=[\lambda_1,\lambda_2]$ fixes $x+\mathbf{Z}$ pointwise, and 
$\eta(\lambda)=[f_1,f_2]$. Since $T_{\tau}(x)'$ is generated by commutators, we conclude that the group $K$ contains the lift $K_x$ of $T_{\tau}(x)'$ that fixes $x+\mathbf{Z}$ pointwise, as a subgroup. This holds for each $x\in \mathbf{R/Z}$. 
It is an elementary exercise to show that the group generated by the groups $\{K_x\mid x\in \mathbf{R/Z}\}$ contains all integer translations. \\

Since $G'$ and $H$ are both lifts of $T_{\tau}'$, by the previous paragraph they coincide. Now $H$ is the preimage under $\eta$ of an index-two subgroup of $T_\tau$, so it has index $2$ in $G = \eta^{-1}(T_\tau)$.
\end{proof}

In order to apply Theorem \ref{thm:zhuang} to $G$, we are only missing the computation of the defect of the rotation number:

\begin{proposition}
\label{prop:rot}

$\rot|_G$ is a quasimorphism of defect $1$.
\end{proposition}

This essentially reduces to the following (see also \cite[Proposition 4.3]{zhuang}):

\begin{lemma}
\label{lem:dense}

$G$ is dense in $\overline{\Homeo^+(\mathbf{R}/\mathbf{Z})}$ with respect to the $C^0$-topology, that is, the topology induced by the supremum norm.
\end{lemma}

\begin{proof}
The group $T_\tau$ is dense in $\Homeo^+(\mathbf{R}/\mathbf{Z})$ in the $C^0$-topology,
since it acts transitively on circularly ordered $n$-tuples in the dense subset $\mathbf{Z}[\tau]/\mathbf{Z}$ for every $n \geq 1$ (Lemma \ref{lem:transT}).
It follows immediately that $G$ is dense in $\overline{\Homeo^+(\mathbf{R/Z})}$.
\end{proof}

\begin{proof}[Proof of Proposition \ref{prop:rot}]
Recall from Example \ref{ex:rot} that $\rot$ has defect $1$ on $\overline{\Homeo^+(\mathbf{R}/\mathbf{Z})}$ \cite[Proposition 2.92]{calegari}. Moreover, $\rot$ is continuous with respect to the $C^0$-topology \cite[Proposition 11.1.6]{ds}. Therefore it follows from Lemma \ref{lem:dense} that $\rot|_G$ also has defect $1$.
\end{proof}

Now we can prove our main Theorem \ref{thm:main}.

\begin{proof}[Proof of Theorem \ref{thm:main}]
Let $G=\overline{T_\tau}$ as before.
The group $T_\tau$ is of type $F_\infty$ by Proposition \ref{prop:Tfin}.
Since $G$ is an extension of the two type $F_\infty$ groups $\mathbf{Z}$ and $T_\tau$, it is also of type $F_\infty$.

By Proposition \ref{prop:Tus}, the group $T_\tau'$ is uniformly simple, in particular it is uniformly perfect, and so $\scl$ vanishes everywhere on $T_\tau'$. By Proposition \ref{prop:index2T}, the index of $T_\tau'$ in $T_\tau$ is $2$, so $\scl$ vanishes everywhere on $T_\tau$. Finally, Proposition \ref{prop:rot} shows that $\rot|_G$ has defect $1$. Therefore Theorem \ref{thm:zhuang} applies, and for every element $g \in G'$ we have $\scl(g) = \rot(g)/2$. \\

Now let $0 < \alpha \in \mathbf{Z}[\tau]$. The translation $f : t \mapsto t+\alpha$ on $\mathbf{R}$ descends to rotation by $\alpha \mod \mathbf{Z}$, so $f \in G$. Moreover $f^2 \in G'$, by Lemma \ref{lemma: lift}, and so
$$\scl(f) = \frac{\scl(f^2)}{2} = \frac{\rot(f^2)}{4} = \frac{\alpha}{2}.$$
\end{proof}

\begin{remark}
The group $T_\tau'$ is also of type $F_\infty$, since this is a quasi-isometry invariant and hence inherited by finite index subgroups \cite{alonso}. Therefore the same proof shows that $\scl(G')$ contains $\mathbf{Z}[\tau]$, and provides a finitely presented (and type $F_{\infty}$) perfect group with algebraic irrational $\scl$. 
\end{remark}

\section{Further examples}

We have given a concrete answer to Question \ref{q}, which gives new insight into the spectrum of $\scl$ on finitely presented groups. It is possible that our proof can be generalized to reach other \emph{metallic ratios}, i.e., quadratic irrationals of the form
$$\lambda := \frac{-n \pm \sqrt{n^2 + 4}}{2}.$$
Indeed, as noted by Cleary at the end of \cite{sqrt2}, most of the arguments establishing finite presentability appear to carry over to that setting. 
On the other hand, it is not clear how to generalize the arguments to other algebraic integers. 
In \cite{cleary}, Cleary cites unpublished work generalizing the theory of regular subdivisions (carried over therein for the golden ratio) to algebraic rings. However, this work remains unpublished. Therefore we end by asking the following more specialized version of Question \ref{q}:

\begin{question}
Does every algebraic integer occur as the $\scl$ of a finitely presented group?
\end{question}

\bibliographystyle{abbrv}
\bibliography{references}

\begin{thebibliography}{10}

\bibitem{alonso}
J.~M. Alonso.
\newblock Finiteness conditions on groups and quasi-isometries.
\newblock {\em Journal of Pure and Applied Algebra}, 95(2):121--129, 1994.

\bibitem{bg}
J.~Barge and E.~Ghys.
\newblock Cocycles d'{E}uler et de {M}aslov.
\newblock {\em Mathematische Annalen}, 294(1):235--265, 1992.

\bibitem{bavard}
C.~Bavard.
\newblock Longueur stable des commutateurs.
\newblock {\em Enseign. Math.(2)}, 37(1-2):109--150, 1991.

\bibitem{brown_fp}
K.~S. Brown.
\newblock Finiteness properties of groups.
\newblock {\em Journal of Pure and Applied Algebra}, 44(1-3):45--75, 1987.

\bibitem{brown_book}
K.~S. Brown.
\newblock {\em Cohomology of groups}, volume~87.
\newblock Springer Science \& Business Media, 2012.

\bibitem{burilloT}
J.~Burillo, B.~Nucinkis, and L.~Reeves.
\newblock Irrational-slope versions of {T}hompson's groups {$T$} and {$V$}.
\newblock {\em arXiv preprint arXiv:2006.02401}, 2020.

\bibitem{burilloF}
J.~Burillo, B.~Nucinkis, and L.~Reeves.
\newblock An irrational-slope {T}hompson's group.
\newblock {\em Publicacions Matem{\`a}tiques}, 65(2):809--839, 2021.

\bibitem{calegari}
D.~Calegari.
\newblock {\em scl}.
\newblock Number~20. Mathematical Society of Japan, 2009.

\bibitem{free}
D.~Calegari.
\newblock Stable commutator length is rational in free groups.
\newblock {\em Journal of the American Mathematical Society}, 22(4):941--961,
  2009.

\bibitem{gog}
L.~Chen.
\newblock Scl in graphs of groups.
\newblock {\em Inventiones mathematicae}, 221(2):329--396, 2020.

\bibitem{sqrt2}
S.~Cleary.
\newblock Groups of piecewise-linear homeomorphisms with irrational slopes.
\newblock {\em The Rocky Mountain Journal of Mathematics}, pages 935--955,
  1995.

\bibitem{cleary}
S.~Cleary.
\newblock Regular subdivision in $\mathbb{Z}[\frac{1+\sqrt{5}}{2}]$.
\newblock {\em Illinois Journal of Mathematics}, 44(3):453--464, 2000.

\bibitem{prox}
{\'S}.~Gal and J.~Gismatullin.
\newblock Uniform simplicity of groups with proximal action.
\newblock {\em Transactions of the American Mathematical Society, Series B},
  4(5):110--130, 2017.

\bibitem{ghys_serg}
{\'E}.~Ghys and V.~Sergiescu.
\newblock Sur un groupe remarquable de diff{\'e}omorphismes du cercle.
\newblock {\em Commentarii Mathematici Helvetici}, 62(1):185--239, 1987.

\bibitem{gromov}
M.~Gromov.
\newblock Asymptotic invariants of infinite groups.
\newblock {\em Geometric group theory}, 2:1--295, 1993.

\bibitem{heuer}
N.~Heuer.
\newblock The full spectrum of scl on recursively presented groups.
\newblock {\em arXiv preprint arXiv:1909.01309}, 2019.
\newblock To appear in \textit{Geometriae Dedicata}.

\bibitem{book}
N.~Heuer.
\newblock Stable commutator length.
\newblock 2021.
\newblock From \textit{Bounded cohomology and Simplicial Volume}, edited by
  C.~Campagnolo, F.~Fournier-Facio, N.~Heuer and M.~Moraschini. To appear in
  the LMS Lecture notes series.

\bibitem{ds}
A.~Katok and B.~Hasselblatt.
\newblock {\em Introduction to the modern theory of dynamical systems}.
\newblock Number~54. Cambridge university press, 1997.

\bibitem{zhuang}
D.~Zhuang.
\newblock Irrational stable commutator length in finitely presented groups.
\newblock {\em J. Mod. Dyn.}, 2(3):499--507, 2008.

\end{thebibliography}

\end{document}